\documentclass[12pt]{amsart}

\usepackage{amssymb,latexsym,eufrak,amsmath,amscd, graphicx}

\usepackage{amsfonts}
\usepackage{amscd}

\renewcommand{\phi}{\varphi}
\renewcommand{\epsilon}{\varepsilon}
\renewcommand{\theta}{\vartheta}

\def\ZZ{{\mathbf Z}}
\def\NN{{\mathbf N}}

\def\CC{{\mathbf C}}
\def\AAA{{\mathbf A}}
\def\RR{{\mathbf R}}
\def\QQ{{\mathbf Q}}

\def\cH{\mathcal{H}}

\def\cJ{\mathcal{J}}

\def\cA{\mathcal{A}}
\def\cF{\mathcal{F}}

\def\cP{\mathcal{P}}

\def\cO{\mathcal{O}}

\def\fra{\mathfrak{a}}

\def\frb{\mathfrak{b}}

\def\frm{\mathfrak{m}}

\def\O{\mathcal{O}}

 \DeclareMathOperator{\Spec}{Spec}
 
 \DeclareMathOperator{\lct}{lct}
 
 \DeclareMathOperator{\ord}{ord}
\DeclareMathOperator{\LCT}{LCT}

\def\.{\cdot}
\def\~{\widetilde}
\def\^{\widehat}

\def\o{\circ}

\newcommand{\llbracket}{[\negthinspace[}
\newcommand{\rrbracket}{]\negthinspace]}

\newtheorem{lemma}{Lemma}[section]
\newtheorem{theorem}[lemma]{Theorem}
\newtheorem{corollary}[lemma]{Corollary}
\newtheorem{proposition}[lemma]{Proposition}

\theoremstyle{definition}

\newtheorem{remark}[lemma]{Remark}

\newtheorem{example}[lemma]{Example}

\theoremstyle{remark}
\newtheorem*{remark*}{Remark}
\newtheorem*{note*}{Note}

\begin{document}

\title{Sequences of LCT-polytopes}

\author[A. Libgober]{Anatoly~Libgober}

\address{Department of Mathematics, University of
Illinois at Chicago, 851 South Morgan Street,
Chicago, IL 60607-7045, USA}
\email{\tt libgober@math.uic.edu}

\author[M. Musta\c{t}\u{a}]{Mircea~Musta\c{t}\u{a}}

\address{Department of Mathematics, University of Michigan,
530 Church Street, 
Ann Arbor, MI 48109, USA}
\email{{\tt mmustata@umich.edu}}

\markboth{A. LIBGOBER AND M. MUSTA\c{T}\u{A}}
{ON SEQUENCES OF LCT-POLYTOPES}

\begin{abstract}
To $r$ ideals on a germ of  smooth variety $X$ one attaches a rational polytope
in $\RR_+^r$ (the \emph{LCT-polytope}) that generalizes the notion of log canonical threshold 
in the case of one ideal. 
We study these polytopes, and prove a strong form of the Ascending Chain Condition in this setting:
we show that if a sequence $(P_m)_{m\geq 1}$ of LCT-polytopes in $\RR_+^r$ converges to a compact 
subset $Q$ in the Hausdorff metric, then $Q=\bigcap_{m\geq m_0}P_m$ for some $m_0$, and 
$Q$ is an LCT-polytope.
\end{abstract}

\thanks{2010\,\emph{Mathematics Subject Classification}.
 Primary 14B05; Secondary 14J17, 14E30.
\newline The first author was patially supported by NSF grant 
DMS-0705050. The second author was partially supported by
 NSF grant DMS-0758454 and
  a Packard Fellowship}

\maketitle

\section{Introduction}

Let $X$ be a smooth algebraic variety over an algebraically closed field $k$,
of characteristic zero. To a nonzero ideal $\fra$ on $X$, and to a point $x$ in the zero
locus of $\fra$ one associates the local log canonical threshold $\lct_x(\fra)$. This  positive
rational number is an invariant of the singularities of $\fra$ at $x$ that plays a fundamental role
in birational geometry (see for example \cite{Kol} and \cite{EM}).

To $r$ ideals $\fra_1,\ldots,\fra_r$ on $X$, and to a point $x$ that lies in the zero locus
of each $\fra_i$ we associate the \emph{LCT-polytope} $\LCT_x(\fra_1,\ldots,\fra_r)$.
This is a rational convex polytope in $\RR_+^r$ that
describes the log canonical thresholds at $x$ of all products $\fra_1^{m_1}\cdots\fra_r^{m_r}$.
More precisely, it consists of those $(\lambda_1,\ldots,\lambda_r)\in\RR_+^r$ such that
the pair $(X,\fra_1^{\lambda_1}\cdots\fra_r^{\lambda_r})$ is log canonical at $x$.
In the case $r=1$, the polytope $\LCT_x(\fra)$ is the segment $[0,\lct_x(\fra)]$.
These polytopes are a special case of the polytopes of quasi-adjunction
introduced and studied by the first author in \cite{alexhodge} and \cite{ample}.
Even if one is only interested in the singularities of one ideal $\fra$, studying the LCT-polytopes
$\LCT(\fra,\frb)$ for various auxiliary ideals $\frb$ gives important information.

Shokurov conjectured in \cite{Sho} that log canonical thresholds in fixed dimension satisfy the Ascending Chain Condition. The conjecture is made in a general setting in which the ambient variety is allowed to have log canonical singularities. Birkar related the general form of the conjecture
 to the Termination of Flips conjecture (see \cite{Birkar} for the precise statement). In the special setting of smooth ambient varieties, Shokurov's conjecture
was proved by de Fernex, Ein and the third author in \cite{dFEM}, building on ideas and results from
\cite{dFM} and \cite{Kol1}.

In this note we consider the Ascending Chain Condition for LCT-polytopes. In particular,
we show that given any sequence of LCT-polytopes in $\RR^r$ 
(corresponding to ideals on smooth $n$-dimensional varieties)
$P_1\subseteq P_2\subseteq\ldots$, the sequence is eventually stationary.
In fact, we prove a much stronger assertion.

We consider the polytopes in $\RR^r$ as elements in the space ${\mathcal H}_r$ of all compact subsets of $\RR^r$ endowed with the Hausdorff metric. This is a complete metric space,
and the subsets lying in a given compact subset $K\subset\RR^r$ form a compact
subspace of ${\mathcal H}_r$.
 It is easy to see that every LCT-polytope as above is contained 
in the cube $[0,n]^r\subseteq\RR^r$. It follows that
 every sequence of LCT-polytopes has a convergent subsequence
to some compact subset $Q\subseteq [0,n]^r$. 

Our main result says that if a sequence of LCT-polytopes $(P_m)_{m\geq 1}$ converges
to the compact set $Q$ in the Hausdorff metric, then there is $m_0$ such that
$Q=\cap_{m\geq m_0}P_m$. Furthermore, $Q$ is a rational convex polytope. In fact, there 
are ideals $\fra_1,\ldots\fra_s\subset K\llbracket x_1,\ldots,x_n\rrbracket$
(for some $s\leq r$ and some field extension $K$ of $k$) 
such that $Q=\LCT(\fra_1,\ldots,\fra_s)$ (under a suitable linear embedding in $\RR^r$). 
If the ground field $k$ has infinite transcendence degree over $\QQ$ (for example, if
$k=\CC$), then we may take $K=k$.

The proof uses the result in \cite{dFEM} about the ACC property of log canonical thresholds
on smooth varieties of fixed dimension. In fact, we use in an essential way also 
the ideas and the constructions in \emph{loc. cit}. We give an introduction
to the basic properties of LCT-polytopes in the following section, emphasizing the
analogy with the case $r=1$. The main theorems are proved in the last section.

\section{Basics of LCT-polytopes}

In this section we present some basic results about  LCT-polytopes. 
 We always work over an algebraically closed field $k$, of characteristic zero.
 We denote by $\RR_+$ the set of nonnegative real numbers, and by $\NN$ the nonnegative 
 integers.
Our ambient space $X$ is either a smooth variety over $k$, or $\Spec(k\llbracket
x_1,\ldots,x_n\rrbracket)$.
We assume that the reader
is familiar with the results about the usual log canonical threshold, for which we refer to
\cite{Kol}, \S 8 for the finite type case, and to \cite{dFM} for the case of formal power series.

Let $X$ be a regular scheme, as above, and $\fra_1,\ldots,\fra_r$ nonzero ideal sheaves on $X$. 
We put 
$$\LCT(\fra_1,\ldots,\fra_r)=\{\lambda=(\lambda_1,\ldots,\lambda_r)\in\RR_+^r\mid 
(X,\fra_1^{\lambda_1}\cdots\fra_r^{\lambda_r})\,\text{is log canonical}\}.$$
We will mostly be concerned with a local variant of this definition: if $x\in X$ is a closed point, then 
$$\LCT_x(\fra_1,\ldots,\fra_r)=\{\lambda=(\lambda_1,\ldots,\lambda_r)\in\RR_+^r\mid 
(X,\fra_1^{\lambda_1}\cdots\fra_r^{\lambda_r})\,\text{is log canonical at}\,x\}.$$
If the ideals $\fra_1,\ldots,\fra_r$ are principal, with $\fra_i=(f_i)$, then we simply write
$\LCT(f_1,\ldots,f_r)$ and $\LCT_x(f_1,\ldots,f_r)$.

The above sets can be explicitly described in terms of a log resolution, as follows. Suppose that
$\pi\colon Y\to X$ is a log resolution of $\fra_1\cdot\ldots\cdot\fra_r$. Recall that this means that
$Y$ is nonsingular, $\pi$ is proper and birational, and we have a simple normal crossings
divisor $\sum_{j=1}^NE_j$ on $Y$ such that 
$$K_{Y/X}=\sum_{j=1}^N\kappa_jE_j,\,\text{and}\, \,
\fra_i\cdot\cO_Y=\cO_Y\left(-\sum_{j=1}^N\alpha_{i,j}E_j\right)\,\text{for}\,1\leq i\leq r.$$
The existence of such a log resolution in the formal power series case is a consequence of the results in \cite{Temkin}.

It follows from the description of log canonical pairs in terms of a log resolution that
$\LCT(\fra_1,\ldots,\fra_r)$ consists precisely of those $\lambda\in\RR_+^r$ such that
\begin{equation}\label{eq1}
\sum_{i=1}^r\alpha_{i,j}\lambda_i\leq \kappa_j+1\,\,\text{for}\,1\leq j\leq N.
\end{equation}
Similarly, $\LCT_x(\fra_1,\ldots\fra_r)$ is cut out by the equations
in (\ref{eq1}) corresponding to those $j$
such that $x\in\pi(E_j)$.

It follows from the above description that both $\LCT(\fra_1,\ldots,\fra_r)$ and 
$\LCT_x(\fra_1,\ldots,\fra_r)$ are rational polyhedra (that is, they are cut out in $\RR^r$
by finitely many affine linear inequalities, with rational coefficients). We call
$\LCT(\fra_1,\ldots,\fra_r)$ and $\LCT_x(\fra_1,\ldots,\fra_r)$ 
the \emph{LCT-polyhedron} of $\fra_1,\ldots,\fra_r$, and respectively,
the \emph{LCT-polyhedron at} $x$ of $\fra_1,\ldots,\fra_r$.

\begin{remark}
The above polyhedra are $r$-dimensional. Indeed, note that they contain the origin, as well
as $\lambda e_i$ for $0<\lambda\ll 1$ (here $e_1,\ldots,e_r$ is the standard basis of
$\RR^r$).
\end{remark}

The following lemma follows immediately from the description of LCT-polyhedra
in terms of a log resolution.

\begin{lemma}\label{lem1}
Given the nonzero ideals $\fra_1,\ldots,\fra_r$, there are closed points 
$x_1,\ldots,x_m\in X$ such that
$$\LCT(\fra_1,\ldots,\fra_r)=\bigcap_{j=1}^m\LCT_{x_j}(\fra_1,\ldots,\fra_r).$$
\end{lemma}

\noindent Because of this lemma, from now on we will focus on the local LCT-polyhedra.

\begin{lemma}\label{lem2}
Let $\fra_1,\ldots\fra_r$ be nonzero ideals on $X$.
\begin{enumerate}
\item[i)] If $x\in {\rm Supp}(V(\fra_i))$, then $\{\lambda_i\mid\lambda\in
\LCT_x(\fra_1,\ldots,\fra_r)\}$ is bounded.
\item[ii)] If $x\not\in {\rm Supp}(V(\fra_r))$, then
$\LCT_x(\fra_1,\ldots,\fra_r)=\LCT_x(\fra_1,\ldots,\fra_{r-1})\times\RR_+$. 
\end{enumerate}
\end{lemma}

\begin{proof}
With the notation in (\ref{eq1}), we see that if $x\in {\rm Supp}(V(\fra_i))$, then there is
$j$ with $\alpha_{i,j}>0$, and such that $x\in\pi(E_j)$. It follows that if $\lambda\in
\LCT_x(\fra_1,\ldots,\fra_r)$, then $\lambda_i\leq (\kappa_j+1)/\alpha_{i,j}$, which gives
i). The assertion in ii) is clear.
\end{proof}

In light of this lemma, it is enough to study $\LCT_x(\fra_1,\ldots,\fra_r)$ for
$x\in\bigcap_i{\rm Supp}(V(\fra_i))$. In this case we see that the LCT-polyhedron at $x$
of $\fra_1,\ldots,\fra_r$ is bounded, hence it is a polytope. We will henceforth refer to it as
the \emph{LCT-polytope at} $x$ of $\fra_1,\ldots,\fra_r$.

\begin{remark} A related construction, giving 
  polyhedra as invariants of  tuples of divisors, 
was used in \cite{ample} and \cite{alexhodge}. Consider
a collection of germs 
$$f_1(x_1,\ldots,x_{n+1}),\ldots,f_r(x_1,\ldots,x_{n+1})$$
of reduced local equations of divisors 
$D_i=V(f_i)$ at a point $P \in X=\CC^{n+1}$, 
that we assume to have isolated non-normal crossings (cf. \cite{ample}). 
With each $\phi \in \O_P$ one associates the top degree form:
\begin{equation}\label{quasiadjform}
\omega_{\phi}(j_1,\ldots,j_r \vert m_1,\ldots,m_r)=
{{{f_1^{{j_1-m_1+1} \over m_1}\cdot\ldots\cdot 
f_r^{{j_r-m_r+1} \over {m_r}}} \phi(x_1,\ldots,x_{n+1}) dx_1 \wedge \ldots
\wedge dx_{n+1}}}
\end{equation} on the unramified covering 
$X_{m_1,\ldots,m_r}$ of $X\smallsetminus\sum_i D_i$ with Galois 
group $\oplus_i \ZZ/m_i\ZZ$. The form $\omega_{\phi}$ 
extends to a holomorphic form on a resolution of singularities of 
a compactification $\overline{X}_{m_1,\ldots,m_r}$ of $X_{m_1,\ldots,m_r}$
if and only if  $({{j_1+1} \over {m_1}},\ldots,{{j_r+1} \over {m_r}})
\in \RR^r$ satisfies a system of linear inequalities, i.e. it  
belongs to a polytope $\cP(\phi \vert f_1,\ldots,f_r)$.
This system can be described
in terms of a log-resolution $\pi\colon Y \rightarrow X$ 
of the principal ideals $(f_1\cdots f_r)$ as above, using 
the resolution of $\overline{X}_{m_1,\ldots,m_r}$ 
given by a resolution of the
quotient singularities of the normalization 
of $\overline{X}_{m_1,\ldots,m_r} \times_X Y$.
This leads to the following explicit collection 
of inequalities describing when  $\lambda=(\lambda_1,\ldots,\lambda_r)\in 
\cP(\phi \vert f_1,\ldots,f_r)$ 
 (cf. \cite[(4)]{alexhodge}):
\begin{equation}\label{quasiadj1}
\sum_{i=1}^r\alpha_{i,j}(1-\lambda_i) 
\leq \kappa_j+1+e_j(\phi)\,\,\text{for}\,1\leq j\leq N.
\end{equation}
Here $\alpha_{i,j}$, $\kappa_j$ are as in (\ref{eq1}), and $e_j(\phi)$ 
is the multiplicity of $\pi^*(\phi)$ along $E_j$.

Vice versa, for a fixed $({{j_1+1} \over {m_1}},\ldots,{{j_r+1} \over {m_r}})$
with $0\leq j_i<m_i$ for all $i$, the set of $\phi\in\cO_P$ 
 such that the given point lies in $\cP(\phi \vert f_1,\ldots,f_r)$
is an ideal $\cA(j_1,\ldots,j_r \vert m_1,\ldots,m_r) \subset \cO_P$ 
(an \emph{ideal of quasi-adjunction}). 

Allowing $\phi$ to run over all elements in $\cO_P$ 
produces a \emph{finite} collection of polytopes in the $[0,1]^r$.
We similarly have a finite collection of ideals of quasi-adjunction.
Moreover, every ideal of quasi-adjunction $\cA$ can be written
as $\cA=\cA(j_1,\ldots,j_r \vert m_1,\ldots,m_r)$ for some point
$({{j_1+1} \over {m_1}},\ldots,{{j_r+1} \over {m_r}})$ that can be 
chosen in the boundary of 
a polytope (\ref{quasiadj1}). The subset of the 
boundary consisting of those
$({{j_1+1} \over {m_1}},\ldots,{{j_r+1} \over {m_r}})$ 
defining a particular $\cA$ is  
a polyhedral subset (face of quasi-adjunction). Therefore one has 
a correspondence between faces $\cF$ 
of the polytopes $\cP(\phi \vert f_1,\ldots,f_r)$
and certain ideals $\cA(\cF)$ in $\cO_P$.

The polytope (\ref{quasiadj1}) corresponding to $\phi=1$ 
coincides with the image of the LCT-polytope (\ref{eq1}) for $\fra_i=(f_i)$
via the affine map $(\lambda_i)\to(1-\lambda_i)$.
 An ideal
of quasi-adjunction  
$\cA(\cF)$ associated to a point 
$({{j_1+1} \over {m_1}},\ldots,{{j_r+1} \over {m_r}}) \in \cF$
coincides with the multiplier ideal of the divisor 
$\sum \mu_i D_i$, where $\mu_i=
1-{{j_i+1} \over {m_i}}-\epsilon$, with $0<\epsilon\ll 1$.
Indeed, strict inequality in the conditions 
(\ref{quasiadj1}) is equivalent to 
$\phi$ being a section of $\pi_*(K_{Y/X}-\lfloor\sum_i(1-\lambda_i)\pi^*(D_i)\rfloor)$.
In the case $r=1$, each polytope (\ref{quasiadj1})
is a segment $[\alpha,1]$, and the face of quasi-adjunction $\alpha$
is a jumping coefficient for the multiplier ideals of $f=f_1$. If the
singularity 
of $f$ at $P$ is isolated, the collection of such $\alpha$ coincides with  
 the subset of the spectrum of the singularity of $f$ in the interval $[0,1]$.
\end{remark}

\begin{example}\label{example1}
If $r=1$, then $\LCT(\fra)=[0,\lct(\fra)]$, and $\LCT_x(\fra)=[0,\lct_x(\fra)]$.
\end{example}

\begin{example}\label{2_0}
If $\fra_i=(x_1^{q_{i,1}}\cdots x_n^{q_{i,n}})\subseteq k[x_1,\ldots,x_n]$, then 
$$\LCT(\fra_1,\ldots,\fra_r)=\left\{\lambda=(\lambda_1,\ldots,\lambda_r)\in\RR_+^r\,\,\vline\,\,
\sum_{i=1}^rq_{i,j}\lambda_i\leq 1\,\text{for}\,1\leq j\leq n\right\}.$$
\end{example}

\begin{example}\label{example2}
One can generalize the previous example to the case of arbitrary monomial ideals. 
This extends Howald's Theorem from \cite{Howald}, which is the case $r=1$.
Suppose that $\fra_1,\ldots,\fra_r$ are nonzero ideals in 
$k[x_1,\ldots,x_n]$ generated by monomials.
Let $P_{\fra_i}$ denote the Newton polyhedron of $\fra_i$, that is,
$P_{\fra_i}$ is the convex hull of $\{u\in \NN^n\mid x^u\in\fra_i\}$. Here, if
$u=(u_1,\ldots,u_n)\in\NN^n$, we denote by $x^u$ the monomial $x_1^{u_1}\cdots x_n^{u_n}$.
By taking a toric resolution of $\fra_1\cdot\ldots\cdot\fra_r$, it is easy to see that
$$\LCT(\fra_1,\ldots,\fra_r)=\LCT_0(\fra_1,\ldots,\fra_r)=\left\{(\lambda_1,\ldots,\lambda_r)\in
\RR_+^r\,\,\vline\,\, e\in \sum_{i=1}^r\lambda_i P_{\fra_i}\right\},$$
where $e=(1,\ldots,1)\in\RR^n$.
\end{example}

\begin{example}\label{example3}
 In the case of plane curves, readily available 
explicit resolutions allow the computation of LCT-polytopes.
In terms of the polytopes of quasi-adjunction considered in \cite{alexhodge}, 
the LCT-polytope is the image of the polytope ``farthest'' 
from the origin along the line $x_1=\ldots=x_r$ under  the
change of variables $(\lambda_i) \to (1-\lambda_i)$.
\par \noindent a) If $f=x$, $g=x-y^2\in k[x, y]$, then
\begin{equation}
 \LCT_0(f,g)= \{(\lambda_1,\lambda_2)\in\RR_+^2 \mid \lambda_1\leq 1,\,
 \lambda_2\leq 1,\, \lambda_1+\lambda_2 \le 3/2 \}.
\end{equation}
\noindent b) If $f=x^2+y^5$, $g=x^5+y^2\in k[x, y]$, then
$\LCT_0(f,g)$ is the intersection of the unit square and of the half planes
\begin{equation}
  10\lambda_1+4\lambda_2 \le 7,\,\,\,\, 4\lambda_1+10 \lambda_2 \le 7.
\end{equation}
\end{example}

\begin{remark}
Even if one is interested in the singularities of an ideal $\fra$, considering the 
LCT-polytopes for several ideals gives interesting information. Suppose, for example,
that $\fra$ is a nonzero ideal on $X$, and $x\in X$ is a closed point in ${\rm Supp}(V(\fra))$.
One defines a function $\phi\colon \RR_+\to \RR_+$ by 
$\phi(t)=\lct_x(\fra\cdot\frm_x^t)^{-1}$, where $\frm_x$ is the ideal defining $x$. 
This is a convex nondecreasing function that encodes useful information about the 
singularities of $\fra$ at $x$.
For example, one can show that the right derivative $\phi'_r(0)$ is equal to
$\lct_x(\fra)^{-1}\cdot\max\frac{\ord_E(\frm_x)}{\ord_E(\fra)}$, where the
maximum is over all divisors $E$ over $X$ that compute $\lct_x(\fra)$.

Note that  $\phi$ is determined by $P:=\LCT_x(\fra,\frm_x)$, and conversely. Indeed,
$\phi(t)=\alpha$ if and only if $\lct(\fra^{1/\alpha}\cdot\frm_x^{t/\alpha})=1$. 
Therefore $\phi(t)$ is characterized by the fact that $(1,t)$
lies on the boundary of
$\phi(t)\cdot P$.
\end{remark}

We record in the following proposition some general properties of LCT-polytopes.
We denote by $e_1,\ldots,e_r$ the standard basis in $\RR^r$. For $\lambda=(\lambda_i)$
and $\mu=(\mu_i)$ in $\RR_+^r$, we put $\lambda\preceq\mu$ if $\lambda_i\leq\mu_i$ for all $i$.
We also put $\lambda\prec\mu$ if  $\lambda_i\leq \mu_i$ for all $i$, with strict inequality when
$\mu_i>0$. 

\begin{proposition}\label{prop1}
 Suppose that $\fra_1,\ldots,\fra_r$
are nonzero ideals on $X$, and $x\in X$ is a closed point such that $x\in {\rm Supp}(V(\fra_i))$  for all $i$.
\begin{enumerate}
\item[i)] If $m_1,\ldots,m_r$ are positive integers, then
the polytope $\LCT_x(\fra_1^{m_1},\ldots,\fra_r^{m_r})$ is the image of  $\LCT_x(\fra_1,\ldots,\fra_r)$
by the map $(u_1,\ldots,u_r)\to (u_1/m_1,\ldots,u_r/m_r)$.
\item[ii)] If $\fra'_i\subseteq\fra_i$ for every $i$, then
$\LCT_x(\fra'_1,\ldots,\fra'_r)\subseteq\LCT_x(\fra_1,\ldots,\fra_r)$.
\item[iii)] $\LCT_x(\fra_1,\ldots,\fra_r)\subseteq\prod_{i=1}^r[0,\lct_x(\fra_i)]\subseteq
[0,n]^r$, where $n=\dim(X)$.
\item[iv)] The simplex
$$\left\{\lambda\in\RR_+^r\mid\sum_{i=1}^r\frac{1}{\lct_x(\fra_i)}\lambda_i\leq 1\right\}$$
is contained in $\LCT_x(\fra_1,\ldots,\fra_r)$.
\item[v)] If $\lambda$, $\lambda'\in\RR_+^r$ are such that $\lambda\preceq\lambda'$, and  
$\lambda'\in \LCT_x(\fra_1,\ldots,\fra_r)$,
then $\lambda\in\LCT_x(\fra_1,\ldots,\fra_r)$.
\end{enumerate}
\end{proposition}

\begin{proof}
All assertions immediately follow from definition, and from familiar facts about singularities of pairs,
see \cite{Kol} and \cite{dFM}.
The assertion in iv) follows from the fact that $\LCT_x(\fra_1,\ldots,\fra_r)$ is convex, and the fact
that the origin, as well as each $\lct_x(\fra_i)e_i$ lies in $\LCT_x(\fra_1,\ldots,\fra_r)$.
\end{proof}

\begin{remark}\label{reduction_principal}
Suppose that $X$ is a nonsingular affine algebraic variety. 
It follows from Proposition~\ref{prop1} iv) that if $f_1,\ldots,f_r\in\cO(X)$, then
$\LCT(f_1,\ldots,f_r)$ is contained in the cube  $[0,1]^r$. On the other hand, if
$\fra_1,\ldots,\fra_r$ are ideals on $X$, and if for every $i$, $g_i\in\fra_i$ is a general linear combination of 
some fixed set of generators of $\fra_i$, then an argument based on Bertini's Theorem as in
\cite[Proposition~9.2.28]{Lazarsfeld} gives
$$\LCT(g_1,\ldots,g_r)=\LCT(\fra_1,\ldots,\fra_r)\cap [0,1]^r.$$
\end{remark}

\begin{remark}\label{completion}
If $\fra_1,\ldots,\fra_r$ are ideals on a smooth variety $X$, and if $x\in X$, then
$\LCT_x(\fra_1,\ldots,\fra_r)=\LCT(\fra_1\cdot\widehat{\cO_{X,x}},\ldots,
\fra_r\cdot\widehat{\cO_{X,x}})$. This follows easily from \cite[Proposition~2.7]{dFM},
that treats the case of log canonical thresholds. Since $\widehat{\cO_{X,x}}\simeq
k\llbracket x_1,\ldots,x_n\rrbracket$, it follows that in order to study the possible
LCT-polytopes in a given dimension $n$, we may restrict to the case when 
$X=\Spec(k\llbracket x_1,\ldots,x_n\rrbracket)$.
\end{remark}

\begin{lemma}\label{intersection_LCT}
If $\fra_1,\ldots,\fra_r$ are nonzero ideals on $X$, and if $\frm_x$ is the ideal defining a closed point 
$x\in X$, then
$$\LCT_x(\fra_1,\ldots,\fra_r)=\bigcap_{q\geq 1}\LCT_x(\fra_1+\frm_x^q,\ldots,\fra_r+\frm_x^q).$$
\end{lemma}

\begin{proof}
The inclusion ``$\subseteq$" is trivial, so let us
suppose that $\lambda=(\lambda_i)$ lies in the above intersection. It is enough to show that
every $\lambda'\in\QQ_+^r$ with $\lambda'\preceq\lambda$ lies in 
$\LCT_x(\fra_1,\ldots,\fra_r)$. Therefore, we may assume that $\lambda\in\QQ_+^r$. 
Choose $N$ such that all $N\lambda_i$ are integers. By assumption, we have 
$\lct((\fra_1+\frm_x^q)^{N\lambda_1}\cdots(\fra_r+\frm_x^q)^{N\lambda_r})\geq 1/N$. 

Let $\tau:=\min\{\lambda_i\mid\lambda_i>0\}$.
Since
the ideals $\fra_1^{N\lambda_1}\cdots\fra_r^{N\lambda_r}$ and
$(\fra_1+\frm_x^q)^{N\lambda_1}\cdots(\fra_r+\frm_x^q)^{N\lambda_r}$ are congruent modulo
$\frm_x^{qN\tau}$, it follows that
$$\lct_x((\fra_1+\frm_x^q)^{N\lambda_1}\cdots(\fra_r+\frm_x^q)^{N\lambda_r})
-\lct_x(\fra_1^{N\lambda_1}\cdots\fra_r^{N\lambda_r})\leq \frac{n}{qN\tau},$$
where $n=\dim(\cO_{X,x})$ (see \cite[Corollary~2.10]{dFM}). We conclude that 
$\lct_x(\fra_1^{\lambda_1}\cdots\fra_r^{\lambda_r})\geq 1-\frac{n}{q\tau}$.
Letting $q$ go to infinity, this gives $\lambda\in\LCT_x(\fra_1\ldots,\fra_r)$.
\end{proof}

The above lemma and the previous remark can be used 
 to reduce proving results about LCT-polytopes on
$\Spec(k\llbracket x_1,\ldots,x_n\rrbracket)$ to proving the similar results on $\AAA^n$.
In order to illustrate this, we give the following

\begin{proposition}\label{inversion}
 If $H\subset X$ is a smooth hypersurface containing $x$,
 and if $\fra_i$ are ideals on $X$
such that all $\fra_i\cO_H$ are nonzero,  then
$$\LCT_x(\fra_1\cO_H,\ldots,\fra_r\cO_H)\subseteq\LCT_x(\fra_1,\ldots,\fra_r).$$
\end{proposition}

\begin{proof}
When $X$ is a nonsingular variety over $k$, this follows easily from Inversion of Adjunction
(see \cite[Theorem~7.5]{Kol}). If $X=\Spec(k\llbracket x_1,\ldots,x_n\rrbracket)$, 
after a change of coordinates we may assume that $H=(x_1=0)$. In this case,
by Lemma~\ref{intersection_LCT} it is enough to prove the proposition when we replace
$\fra_i$ by $\fra_i+\frm_x^q$. Since there are ideals $\fra'_i$ in $k[x_1,\ldots,x_n]$ such that
$\fra_i+\frm_x^q=\fra'_i\cdot k\llbracket x_1,\ldots,x_n\rrbracket$, we conclude using the case
of ideals in $\AAA^n$ via Remark~\ref{completion}.
\end{proof}

\begin{remark}\label{rem_mixed_multiplier}
If $X$ is a nonsingular variety over $k$, it is sometimes convenient to phrase
the description of $\LCT_x(\fra_1,\ldots,\fra_r)$ in the language of mixed multiplier ideals,
for which we refer to \cite[Chapter~9]{Lazarsfeld}. Recall that the pair
$(X,\fra_1^{\lambda_1}\cdots\fra_r^{\lambda_r})$ is klt at $x\in X$ if and only if the mixed multiplier ideal $\cJ(X,\fra_1^{\lambda_1}\cdots\fra_r^{\lambda_r})$ is not contained in the ideal
$\frm_x$ defining $x$. We deduce using the definition of the LCT-polytopes that 
$\lambda\in\LCT_x(\fra_1,\ldots,\fra_r)$ if and only if for every $\mu=(\mu_i)\in\RR_+^r$ with
$\mu\prec\lambda$, we have $\cJ(X,\fra_1^{\mu_1}\cdots\fra_r^{\mu_r})
\not\subseteq\frm_x$.
\end{remark}

The following proposition is the generalization to the case $r>1$
of \cite[Proposition~8.19]{Kol}. As above, we denote by 
$\frm_x$ the ideal defining the closed point $x\in X$.

\begin{proposition}\label{prop2}
Let $\frb,\fra_1,\ldots,\fra_r$ be nonzero ideals on $X$, with $\dim(X)=n$. 
If $\lambda=(\lambda_j)\in \LCT_x(\fra_1,\ldots,\fra_r)$, and $N$ 
is a positive integer such that
$\fra_i+\frm_x^N=\frb+\frm_x^N$ for some $i$, then
$$\lambda-\min\{n/N,\lambda_i\}e_i \in\LCT_x(\fra_1,\ldots,\frb,\ldots,\fra_r),$$
where $\frb$ appears on the $i^{\rm th}$ component.
\end{proposition}

\begin{proof}
By Lemma~\ref{lem2}, we may assume that all $\fra_i$ vanish at $x$.
After replacing $\fra_i$ by $\fra_i+\frm_x^N$, we may also assume that $\fra_i=\frb+\frm_x^N$.
Arguing as in the proof of  Proposition~\ref{inversion}, we see that it is enough to prove
the statement when $X$ is a smooth variety over $k$.
In this case it is convenient to use
the language of mixed multiplier ideals, see Remark~\ref{rem_mixed_multiplier}.
Let us consider any $\mu=(\mu_j)\in\RR_+^r$, with $\mu\prec\lambda$, so by assumption
  the mixed multiplier ideal
$\cJ(X, \fra_1^{\mu_1}\cdots\fra_r^{\mu_r})$ is not contained in $\frm_x$.

By the Summation Theorem 
(for the version that we need, see \cite[Corollary~4.2]{JM}) we have
$$\cJ(X,\fra_1^{\mu_1}\cdots (\frb+\frm_x^N)^{\mu_i}\cdots\fra_r^{\mu_r})=
\sum_{\alpha+\beta=\mu_i}\cJ(X,\fra_1^{\mu_1}\cdots
\frb^{\alpha}\frm_x^{N\beta}\cdots\fra_r^{\mu_r}).$$
It follows that for some $\alpha$, $\beta\geq 0$ with $\alpha+\beta=\mu_i$ we have
$$\cJ(X,\fra_1^{\mu_1}\cdots
\frb^{\alpha}\frm_x^{N\beta}\cdots\fra_r^{\mu_r})\not\subseteq\frm_x.$$
If $\mu_i>\frac{n}{N}$,  then using
$\cJ(\frm_x^n)\subseteq\frm_x$ we deduce 
$N\beta<n$, and therefore
$$\left(\mu_1,\ldots,\mu_i-\frac{n}{N},\ldots,\mu_r\right)\in\LCT_x(\fra_1,\ldots,\frb,\ldots,\fra_r).$$
We conclude that $\mu-\min\{n/N,\mu_i\}e_i\in \LCT_x(\fra_1,\ldots,\frb,\ldots,\fra_r)$
(note that by hypothesis $(\mu_1,\ldots,0,\ldots,\mu_r)\in \LCT_x(\fra_1,\ldots,\fra_i,\ldots\fra_r)$,
which is equivalent to $(\mu_1,\ldots,0,\ldots,\mu_r)\in \LCT_x(\fra_1,\ldots,\frb,\ldots\fra_r)$).
Since this holds for every $\mu\prec\lambda$, we get the conclusion of the proposition.
\end{proof}

An iterated application of the proposition gives the following result
improving Lemma~\ref{intersection_LCT}.

\begin{corollary}\label{cor1}
Let $\fra_i$, $\frb_i$ be ideals on $X$, for $1\leq i\leq r$, and let $N$ be a positive integer such that
$\fra_i+\frm_x^N=\frb_i+\frm_x^N$ for all $i$. If $\lambda=(\lambda_i)\in\LCT_x(\fra_1,\ldots,\fra_r)$, then
$\lambda'=(\lambda'_i)\in\LCT_x(\frb_1,\ldots,\frb_r)$, where $\lambda'_i=\max\left\{\lambda_i-\frac{n}{N},
0\right\}$ for all $i$.
\end{corollary}

Recall that on the space ${\mathcal H}_r$ of all nonempty compact subsets in $\RR^r$ we have the Hausdorff metric, defined as follows. If $K\subset\RR^r$ is an arbitrary compact set, for every $x\in\RR^r$
we put $d(x,K)=\min_{y\in K}d(x,y)$, where $d(x,y)$ denotes the Euclidean distance between
$x$ and $y$. The Hausdorff distance between two compact sets $K_1$ and $K_2$ is defined by
$$\delta(K_1,K_2):=\max\{\max_{x\in K_1}d(x,K_2),\max_{x\in K_2}d(x,K_1)\}.$$
The set of all nonempty compact subsets of $\RR^r$ thus becomes 
a complete metric space. Furthermore, the subspace of $\cH_r$
consisting of all compact subsets of a fixed compact set 
$K$ in  $\RR^r$ is compact.
For some basic facts about the Hausdorff metric, see \cite[p.281]{Munkres}.
Using this notion, we deduce from Corollary~\ref{cor1} the next

\begin{corollary}\label{cor2}
Suppose that $\fra_i$, $\frb_i$ are ideals on $X$, and $x\in X$ lies in
$\bigcap_i{\rm Supp}(V(\fra_i))$. If $N$ is a positive integer such that
$\fra_i+\frm_x^N=\frb_i+\frm_x^N$ for all $i$, then
$$\delta(\LCT_x(\fra_1,\ldots,\fra_r),\LCT_x(\frb_1,\ldots,\frb_r))\leq\frac{n\sqrt{r}}{N}.$$
\end{corollary}

\begin{example}\label{ex_11}
Let $\fra_1,\ldots,\fra_r$ be proper nonzero ideals on $X=\Spec(k\llbracket x_1,\ldots,x_n\rrbracket)$.
If $\frb_1,\ldots,\frb_r$ are the inverse images of these ideals on 
$X'=\Spec(k\llbracket x_1,\ldots,x_n,y\rrbracket)$ via the canonical projection, then
$\LCT(\frb_1+(y^d),\frb_2,\ldots,\frb_r)$ is equal to
\begin{equation}\label{eq_ex_11}
\{(\lambda_1+t,\lambda_2,\ldots,\lambda_r)
\mid (\lambda_1,\ldots,\lambda_r)\in\LCT(\fra_1,\ldots,\fra_r), 0\leq t\leq 1/d\}.
\end{equation}
Indeed, note first that by Lemma~\ref{intersection_LCT} (or Corollary~\ref{cor1}), it is enough to prove the above assertion when we replace each $\fra_i$ by $\fra_i+(x_1,\ldots,x_n)^{\ell}$, for all
$\ell\geq 1$. It follows from Remark~\ref{completion} that it is enough to prove the similar equality
when the $\fra_i$ are nonzero ideals on $\Spec(k[x_1,\ldots,x_n])$ vanishing at the origin, we have
$\frb_i=\fra_i\cdot k[x_1,\ldots,x_n,y]$, and we compute the LCT-polytopes at the origin.
In this case it is again convenient to use the language of mixed multiplier ideal sheaves.
Recall that by Remark~\ref{rem_mixed_multiplier}, we have $\lambda\in 
\LCT_0(\fra_1,\ldots,\fra_r)$ if and only if 
for every 
$\mu=(\mu_i)\in\RR_+^r$ with $\mu\prec\lambda$, we have
$\cJ(\AAA^n,\fra_1^{\mu_1}\cdots\fra_r^{\mu_r})\not\subseteq(x_1,\ldots,x_n)$.
It follows from the Summation Theorem (see \cite[Corollary~4.2]{JM})
that for every $\mu_1,\ldots,\mu_r\in\RR_+$, we have
$$\cJ(\AAA^{n+1},(\frb_1+(y^d))^{\mu_1}\frb_2^{\mu_2}\cdots\frb_r^{\mu_r})=
\sum_{\alpha+\beta=\mu_1}\cJ(\AAA^{n+1}, \frb_1^{\alpha}y^{d\beta}\frb_2^{\mu_2}\cdots\frb_r^{\mu_r})$$
$$=\sum_{\alpha+\beta=\mu_1}(y^{\lfloor d\beta\rfloor})\cdot
\cJ(\AAA^n, \frb_1^{\alpha}\frb_2^{\mu_2}\cdots\frb_r^{\mu_r}),$$
where the second equality follows from \cite[Remark~9.5.23]{Lazarsfeld}.
Therefore, this ideal is not contained in $(x_1,\ldots,x_n,y)$ if and only 
there is $\beta\in\RR_+$ with $\beta_1<1/d$ such that
$\cJ(\AAA^n,\frb_1^{\mu_1-\beta}\frb_2^{\mu_2}\cdots\frb_r^{\mu_r})$
is not contained in $(x_1,\ldots,x_n)$.
The description in (\ref{eq_ex_11}) easily follows.
\end{example}

\section{Limits of LCT-polytopes}

Recall that by Remark~\ref{completion}, 
 in order to study the possible
LCT-polytopes in a given dimension $n$, we may restrict to the case when 
$X=\Spec(k\llbracket x_1,\ldots,x_n\rrbracket)$. Of course, in this case it is not necessary to
include the closed point in the notation.

\begin{remark}\label{rem_k_and_K}
Note that if $k\subset K$ is a field extension of algebraically closed fields, and if 
$\fra_1,\ldots,\fra_r$ are nonzero proper ideals in $k\llbracket x_1,\ldots,x_n\rrbracket$, and if
we put $\fra'_i=\fra_i\cdot K\llbracket x_1,\ldots,x_n\rrbracket$, then
$\LCT(\fra_1,\ldots,\fra_r)=\LCT(\fra'_1,\ldots,\fra'_r)$. Indeed, by
Lemma~\ref{intersection_LCT} it is enough to show that for all $N\geq 1$ we have
\begin{equation}\label{eq_k_and_K}
\LCT(\fra_1+\frm^N,\ldots,\fra_r+\frm^N)=\LCT(\fra'_1+(\frm')^N,\ldots,\fra'_r+(\frm')^N),
\end{equation}
where $\frm$ and $\frm'$ are the maximal ideals in $k\llbracket x_1,\ldots,x_n\rrbracket$
and respectively, $K\llbracket x_1,\ldots,x_n\rrbracket$. Let us fix $N$.
There are ideals $\frb_i$ in $k[x_1,\ldots,x_n]$ such that
$\frb_i\cdot k\llbracket x_1,\ldots,x_n\rrbracket=\fra_i+\frm^N$ for every $i$. 
If $\frb'_i=\frb_i\cdot K[x_1,\ldots,x_n]$, then $\frb'_i\cdot
K\llbracket x_1,\ldots,x_n\rrbracket=\fra'_i$. 
It is easy to see that $\LCT_0(\frb_1,\ldots,\frb_r)=\LCT_0(\frb'_1,\ldots\frb'_r)$,
using a log resolution of $\frb_1\cdot\ldots\cdot\frb_r$ to compute the left-hand side of the equality,
and the base-extension of this log resolution to $\Spec(K)$ to compute the right-hand side
(see for example \cite[Proposition~2.9]{dFM} for the case of one ideal). The assertion in
(\ref{eq_k_and_K}) is now a consequence of Remark~\ref{completion}.
Therefore every LCT-polytope of 
ideals in $k\llbracket x_1,\ldots,x_n\rrbracket$ is an LCT-polytope of ideals in
$K\llbracket x_1,\ldots,x_n\rrbracket$. 
\end{remark}

\begin{remark}\label{converse}
If $k$ is an algebraically closed field having infinite transcendence degree over $\QQ$
(for example, $k=\CC$), then every LCT-polytope of $r$ ideals in some 
$K\llbracket x_1,\ldots,x_n\rrbracket$, where $K$ is an algebraically closed field extension of 
$k$, can be realized as the
LCT-polytope of $r$ ideals in $k\llbracket x_1,\ldots,x_n\rrbracket$. Indeed, suppose that
$P=\LCT(\fra_1,\ldots,\fra_r)$, with $\fra_1,\ldots,\fra_r$ proper nonzero ideals in 
$K\llbracket x_1,\ldots,x_n\rrbracket$. Since each $\fra_i$ is finitely generated, we can find
an algebraically closed subfield $L\subset K$ of countable transcendence degree over $\QQ$,
and ideals $\frb_i$ in $L\llbracket x_1,\ldots,x_n\rrbracket$ such that 
$\fra_i=\frb_i\cdot K\llbracket x_1,\ldots,x_n\rrbracket$ for every $i$. Using the fact that
$k$ has infinite transcendence degree over $\QQ$, we can find an embedding
$L\hookrightarrow k$. If $\frb'_i=\frb_i\cdot k\llbracket x_1,\ldots,x_n\rrbracket$, we
deduce from the previous remark that $\LCT(\fra_1,\ldots,\fra_n)=\LCT(\frb'_1,\ldots,\frb'_n)$.
\end{remark}

By Proposition~\ref{prop1} iv),
all LCT-polytopes corresponding to $r$ proper nonzero ideals in 
$k\llbracket x_1,\ldots,x_n\rrbracket$
are contained in the compact set $[0,n]^r$.
Therefore every sequence of LCT-polytopes has a convergent
subsequence (in the Hausdorff metric). 
Our goal is to show 
that the limit is again an LCT-polytope, corresponding to possibly fewer than $r$ ideals.
Furthermore, we prove that in this case, the limit is 
equal to the intersection of all but finitely many of the given
LCT-polytopes.

\begin{theorem}\label{limit}
If $P_m=\LCT(\fra_1^{(m)},\ldots,\fra_r^{(m)})$ for $m\geq 1$, where the $\fra_i^{(m)}$ are proper nonzero ideals in $k\llbracket x_1,\ldots,x_n\rrbracket$, and if the $P_m$ converge
in the Hausdorff metric  to a compact
set $Q\subseteq\RR^r$, then $Q$ is again an LCT-polytope. More precisely, if 
$I$ is the set of those $i\leq r$ such that $Q\not\subseteq (x_i=0)$, then we can find proper
nonzero ideals
$\fra_1,\ldots,\fra_s$ in $K\llbracket x_1,\ldots,x_n\rrbracket$, with $s=\#I$ and $K$ an
algebraically closed field extension of $k$, such that
$Q=j_I(\LCT(\fra_1,\ldots,\fra_s))$, where $j_I\colon\RR^s\hookrightarrow\RR^r$ is the inclusion corresponding to the coordinates in $I$.
\end{theorem}

\begin{remark}
We make the convention that the LCT-polytope of an empty set of ideals consists of $\{0\}$.
In the context of Theorem~\ref{limit}, it can happen that $s=0$, in which case $Q$ consists
of the origin in $\RR^r$.
\end{remark}

\begin{remark}
It follows from Remark~\ref{converse} that if the transcendence degree of $k$ over $\QQ$
is infinite, then in Theorem~\ref{limit} we may take $K=k$.
\end{remark}

\begin{theorem}\label{strong_ACC}
If $(P_m)_{m\geq 1}$ and $Q$ are as in Theorem~\ref{limit}, then there is $m_0$ such that
$Q=\bigcap_{m\geq m_0}P_m$.
\end{theorem}

This result can be considered as a strong form of the Ascending Chain Condition for
LCT-polytopes. In fact, it immediately gives

\begin{corollary}\label{ACC}
If $P_m=\LCT(\fra_1^{(m)},\ldots,\fra_r^{(m)})$ for $m\geq 1$, where the $\fra_i^{(m)}$ are proper nonzero ideals in $k\llbracket x_1,\ldots,x_n\rrbracket$, and if $P_1\subseteq P_2\subseteq\cdots$,
then this sequence is eventually stationary.
\end{corollary}

\begin{proof}
It is enough to find a subsequence that is eventually stationary.
Since $P_m\subseteq [0,n]^r$ for all $m$, we deduce that 
after passing to a subsequence, we may assume that the $P_m$ converge
to some $Q$ in the Hausdorff metric. Theorem~\ref{strong_ACC} implies that there is
$m_0$ such that $Q=\bigcap_{m\geq m_0}P_m$. On the other hand, it is easy to see
that in our case $\bigcup_{m\geq 1} P_m\subseteq Q$ (see, for example, Lemma~\ref{lem3_1} iii)
below). This gives $P_m=Q$ for every $m\geq m_0$. 
\end{proof}

For the proof of Theorems~\ref{limit} and \ref{strong_ACC} we will need a couple of lemmas.
The first one gives some easy properties of Hausdorff convergence that we will need.
We denote by $d(\cdot,\cdot)$ the Euclidean distance in $\RR^r$, and by $\delta(\cdot,\cdot)$
the Hausdorff metric on the space ${\mathcal H}_r$ of all nonempty compact 
subsets of $\RR^r$.
\begin{lemma}\label{lem3_1}
Let $(K_m)_{m\geq 1}$ be a sequence of compact subsets in $\RR^r$, converging in the Hausdorff metric
to the compact subset $K$.
\begin{enumerate}
\item[i)] If $C\subseteq\RR^r$ is closed, and $K_m\subseteq C$ for all $m$, then
$K\subseteq C$.
\item[ii)] If $u_m\in K_m$, and $(u_m)_{m\geq 1}$ converges to $u\in\RR^r$, then $u\in K$. 
\item[iii)] $\bigcap_mK_m\subseteq K$.
\end{enumerate}
\end{lemma}

\begin{proof}
The assertion in i) follows easily from definition.
For ii), note that if $u\not\in K$, then there is a ball $B(u,\epsilon)$ centered at $u$, and of radius 
$\epsilon>0$
that does not intersect $K$. By assumption, there is $m_0$ such that $\delta(K_m,K)<
\epsilon/2$ for all $m\geq m_0$. For such $m$, since $u_m\in K_m$, we have
$d(u_m,K)<\epsilon/2$, hence we can find $w_m\in K$ such that $d(u_m,w_m)
<\epsilon/2$. On the other hand, after possibly enlarging $m_0$, we may assume that 
 $d(u_m,u)<\epsilon/2$
for $m\geq m_0$. Therefore 
$$d(u,w_m)\leq d(u,u_m)+d(u_m,w_m)<\epsilon/2+\epsilon/2=\epsilon,$$
contradicting the fact that $B(u,\epsilon)\cap K=\emptyset$.
This proves ii), and the assertion in iii) is a special case.
\end{proof}

For a proper nonzero ideal $\fra$ in $k\llbracket x_1,\ldots,x_n\rrbracket$, its order
$\ord(\fra)$ is the largest nonnegative integer $d$ such that $\fra$ is contained
in the $d^{\rm th}$ power of the maximal ideal $\frm$.
Recall the following estimates for the log canonical threshold in terms of the order:
\begin{equation}\label{order}
\frac{1}{\ord(\fra)}\leq\lct(\fra)\leq\frac{n}{\ord(\fra)}
\end{equation}
(the first inequality reduces to
the case $n=1$ via Proposition~\ref{inversion}, while the second inequality
follows from $\lct(\fra)\leq\lct(\frm^{\ord(\fra)})=n/\ord(\fra)$).

\begin{lemma} \label{lem3_2}
With the notation in Theorem~\ref{limit}, the following are equivalent:
\begin{enumerate}
\item[i)] $Q\subseteq (x_i=0)$.
\item[ii)] $\lim_{m\to\infty}\ord(\fra_i^{(m)})=\infty$.
\item[iii)] The set $\{\ord(\fra_i^{(m)})\mid m\geq 1\}$ is unbounded.
\end{enumerate}
\end{lemma}

\begin{proof}
Suppose first that $Q\subseteq (x_i=0)$. For every $m$ we have 
$\lct(\fra_i^{(m)})\cdot e_i\in P_m$, where $e_1,\ldots,e_r$ is the standard
basis of $\RR^r$. It follows from Lemma~\ref{lem3_1} ii)
that every limit point of the sequence $\left(\lct(\fra_i^{(m)})\cdot e_i\right)_{m\geq 1}$
lies in $Q$. Therefore $\lim_{m\to\infty}\lct(\fra_i^{(m)})=0$, and ii) follows
from the first inequality in (\ref{order}).

Since the implication ii)$\Rightarrow$iii) is trivial, in order to finish the proof of the lemma
it is enough to prove iii)$\Rightarrow$i). Suppose that 
$\lambda=(\lambda_1,\ldots,\lambda_r)\in Q$, and $\lambda_i>0$.
We can find $m_0$ such that $\delta(P_m,Q)<\lambda_i/2$ for all $m\geq m_0$.
For every such $m$, we can find $w^{(m)}=(w^{(m)}_1,\ldots,w^{(m)}_r)\in P_m$ such that
$d(w^{(m)},\lambda)<\lambda_i/2$. In particular, $w^{(m)}_i>\lambda_i/2$.
Since $w^{(m)}\in P_m$, we see using the
second inequality in (\ref{order})
that for all $m\geq m_0$
$$\frac{\lambda_i}{2}<w^{(m)}_i\leq\lct(\fra_i^{(m)})\leq\frac{n}{\ord(\fra_i^{(m)})}.$$
This contradicts iii).
\end{proof}

The main ingredient in the proof of Theorems~\ref{limit} and \ref{strong_ACC} is the 
generic limit construction from \cite{Kol1} and 
\cite{dFEM}. Let $(\fra_1^{(m)})_m,\ldots,(\fra_r^{(m)})_m$ be sequences as in
Theorem~\ref{limit}.
In order to simplify the notation, let us relabel the sequences such that
the set $I$ in the theorem is equal to $\{1,\ldots,s\}$.
Associated to the $s$ sequences 
 $(\fra_i^{(m)})_{m\geq 1}$, with $1\leq i\leq s$, 
 we get $s$ \emph{generic limits} $\fra_1,\ldots,\fra_s$.
 These are ideals in 
 $K\llbracket x_1,\ldots,x_n\rrbracket$, where $K$ is a suitable 
 algebraically closed field extension of $k$.
It follows from Lemma~4.3 in \cite{dFEM} and the above Lemma~\ref{lem3_2} that 
all $\fra_i$ are nonzero. Furthermore, since every $\fra_i^{(m)}$ is contained in the
maximal ideal, the same holds for the ideals $\fra_i$. 
The fundamental property of the generic limit construction is that there is a strictly increasing sequence $(m_{\ell})_{\ell}$ such that for every nonnegative rational numbers
$w_1,\ldots,w_s$ we have
\begin{equation}\label{eq_limit}
\lim_{\ell\to\infty}\lct((\fra_1^{(m_{\ell})})^{w_1}\cdots(\fra_s^{(m_{\ell})})^{w_s})
=\lct(\fra_1^{w_1}\cdots\fra_s^{w_s})
\end{equation}
(see \cite[Corollary~4.5]{dFEM}).

\begin{remark}
The construction in \cite{dFEM} deals with only two sequences of ideals, but as pointed out
in \emph{loc. cit.}, everything generalizes in an obvious way to any finite number of sequences.
We also note that the field $K$ given in \emph{loc. cit.} is not algebraically closed, but since
we are only interested in (\ref{eq_limit}), we can simply extend the generic limit ideals to
an algebraic closure. The equation (\ref{eq_limit}) is stated in \emph{loc. cit.} only for
integers $w_1,\ldots,w_s$. On the other hand, if the $w_i$ are rational numbers, and if
$N$ is a positive integer such that all $Nw_i\in\ZZ$, the formula for $(Nw_1,\ldots,Nw_s)$
implies the one for $(w_1,\ldots,w_s)$ by rescaling.
\end{remark}

We isolate in the following lemma the key argument needed for the proofs
of Theorems~\ref{limit} and \ref{strong_ACC}. We use the notation in those theorems,
as well the notation for the generic limit ideals introduced above.

\begin{lemma}\label{lem3_3}
If $\lambda\in \LCT(\fra_1,\ldots,\fra_s)\cap\QQ^s$, then there are infinitely
many $m$ such that $j_I(\lambda)\in P_m$.
\end{lemma}

\begin{proof}
Write $\lambda=(\lambda_1,\ldots,\lambda_s)$, hence by assumption
$\lct(\fra_1^{\lambda_1}\cdots\fra_s^{\lambda_s})\geq 1$. Fix a positive integer $N$
such that $N\lambda_i\in\ZZ$ for every $i$.
Consider the set
$$\Gamma:=\{\lct((\fra_1^{(m)})^{N\lambda_1}\cdots
(\fra_s^{(m)})^{N\lambda_s})\mid m\in \ZZ_{>0}\}.$$
Since the elements of $\Gamma$ are log canonical thresholds
of ideals on $\Spec(k\llbracket x_1,\ldots,x_n\rrbracket)$,
it follows from \cite[Theorem~5.1]{dFEM} that $\Gamma$ satisfies ACC, that is, it contains no infinite
strictly increasing sequences. On the other hand,  (\ref{eq_limit}) shows that 
$\frac{1}{N}\lct(\fra_1^{\lambda_1}\cdots\fra_s^{\lambda_s})$ lies in the closure of
$\Gamma$. We deduce that there are infinitely many $m$ such that
$$\lct((\fra_1^{(m)})^{N\lambda_1}\cdots
(\fra_s^{(m)})^{N\lambda_s})\geq \frac{1}{N}\lct(\fra_1^{\lambda_1}\cdots
\fra_s^{\lambda_s})\geq
\frac{1}{N}.$$
Therefore $j_I(\lambda)\in P_m$ for all such $m$.
\end{proof}

We can now give the proofs of our main results.

\begin{proof}[Proof of Theorem~\ref{limit}]
With the above notation, it is enough to show that we have
$Q=j_I(\LCT(\fra_1,\ldots,\fra_s))$
(of course, we may assume that $s\geq 1$, as otherwise there is nothing to prove). Note first that Lemma~\ref{lem3_3} gives the
inclusion $j_I(\LCT(\fra_1,\ldots,\fra_s))\subseteq Q$. Indeed, since 
$\LCT(\fra_1,\ldots,\fra_s)\cap\QQ^s$ is dense in $\LCT(\fra_1,\ldots,\fra_s)$, and $Q$ is closed,
it is enough to prove the inclusion $j_I(\LCT(\fra_1,\ldots,\fra_s)\cap\QQ^s)\subseteq Q$, and this follows from
the lemma (note that by Lemma~\ref{lem3_1} iii), the intersection of infinitely many of the 
$P_m$ is contained in $Q$).

We now prove the reverse inclusion: suppose that $u=(u_1,\ldots,u_r)\in Q$ (hence
$u_i=0$ for $i>s$), and let us show that
$(u_1,\ldots,u_s)\in \LCT(\fra_1,\ldots,\fra_s)$. 
Note first that by Lemma~\ref{lem3_1} i), we have $Q\subseteq\RR_+^r$.
Fix $\epsilon>0$, and let us choose
$w=(w_1,\ldots,w_s)\in \QQ_+^s$ such that $w_i\leq u_i$ for all $i$, with strict inequality if
$u_i>0$, and such that $(u_i-w_i)<\epsilon$ for all $i$. We will show that in this case
$\lct(\fra_1^{w_1}\cdots\fra_s^{w_s})\geq 1$. Since this holds for every $\epsilon>0$,
we get $\lct(\fra_1^{u_1}\cdots\fra_s^{u_s})\geq 1$, that is, $u\in
j_I(\LCT(\fra_1,\ldots,\fra_s))$.

Let $(m_{\ell})$ be a strictly increasing sequence such that (\ref{eq_limit}) holds.
We can choose 
$q$ such that for all $m\geq q$ we have
$\delta(P_m,Q)<\min\{u_i-w_i\mid u_i>0\}$. For every such $m$, let us choose $v_m\in P_m$
with $d(v_m,u)<\min\{u_i-w_i\mid u_i>0\}$. We may assume that $v_m\in\QQ^r$.
Since $v_m=(v_{m,1},\ldots,v_{m,r})\in P_m$, we have
$\lct((\fra_1^{(m)})^{v_{m,1}}\cdots(\fra_r^{(m)})^{v_{m,r}})\geq 1$.
On the other hand, by construction $w_i\leq v_{m,i}$ for every $i\leq s$, hence
$\lct((\fra_1^{(m)})^{w_1}\cdots(\fra_s^{(m)})^{w_s})\geq 1$ for all $m\geq q$.
Therefore (\ref{eq_limit}) implies 
$\lct(\fra_1^{w_1}\cdots\fra_s^{w_s})\geq 1$, completing the proof.
\end{proof}

\begin{proof}[Proof of Theorem~\ref{strong_ACC}]
It is enough to show that there is $m_0$ such that $Q\subseteq P_m$ for all
$m\geq m_0$. Indeed, in this case $Q\subseteq\bigcap_{m\geq m_0}P_m\subseteq Q$,
where the second inclusion follows from Lemma~\ref{lem3_1} iii).

Let us assume that this is not the case. After possibly replacing the sequence $(P_m)_{m\geq 1}$
by a subsequence, we may assume that $Q\not\subseteq P_m$ for any $m$. Note that by
Theorem~\ref{limit}, $Q$ is a rational polytope, so it is the convex hull of its vertices, which
lie in $\QQ^r$. Furthermore, by the above proof, each such vertex lies in
$j_I(P(\fra_1,\ldots,\fra_s))$; hence by Lemma~\ref{lem3_3}, it lies in infinitely many
$P_m$. After replacing the sequence $(P_m)_{m\geq 1}$ by a subsequence, and after doing this 
for all vertices of $Q$, we conclude that all vertices of $Q$ lie in $P_m$ for all $m$. Therefore
$Q\subseteq P_m$ for all $m$, a contradiction. This concludes the proof of the theorem.
\end{proof}

\begin{example}
It follows from Example~\ref{ex_11} that if $\fra_1,\ldots,\fra_r$ are proper
nonzero ideals in $k\llbracket x_1,\ldots,x_n\rrbracket$, then
$\LCT(\fra_1,\ldots,\fra_r)$ is the intersection of a sequence 
$P_1\supset P_2\supset\ldots$ that is \emph{not} eventually stationary, 
where each $P_i$ is the $\LCT$-polytope of $r$ proper nonzero ideals in
$k\llbracket x_1,\ldots,x_n,y\rrbracket$.
\end{example}

\bigskip

\begin{remark}
If in Theorem~\ref{limit} we have $P_m=\LCT(f^{(m)}_1,\ldots,f^{(m)}_r)$ with the 
$f_i^{(m)}$ nonzero elements in the maximal ideal of $k\llbracket x_1,\ldots,x_n\rrbracket$,
then one can obtain $Q$ as (the linear embedding of) $\LCT(f_1,\ldots,f_s)$,
with $f_i$ nonzero  elements in the maximal ideal of some $K\llbracket x_1,\ldots,x_n\rrbracket$.
Indeed, one can modify the construction in \cite{dFEM} by replacing the Hilbert schemes
parametrizing all ideals in quotient rings $k[x_1,\ldots,x_n]/(x_1,\ldots,x_n)^d$ with parameter spaces for principal ideals in these rings (when $r=1$, this is done in 
\cite{Kol1}).
\end{remark}

Since the set of all log canonical thresholds $\lct(f)$, with $f\in k\llbracket x_1,\ldots,x_n\rrbracket$
satisfies ACC, it follows that there is a largest such invariant that is $<1$. Finding this value
for arbitrary $n$ is an open problem. For example, it is well-known that this value is equal to
$\frac{5}{6}$ if $n=2$. Indeed, if $f\in k\llbracket x,y\rrbracket$ has order $\geq 3$, then
we have
$\lct(f)\leq\frac{2}{3}$ by (\ref{order}). On the other hand, if the multiplicity
of $f$ at $0$ is two, then $f$ is formally equivalent to $x^2+y^m$, for some $m\geq 2$, and 
$\lct_0(x^2+y^m)=\frac{1}{2}+\frac{1}{m}$ (see \cite[\S 9.3.C]{Lazarsfeld}).
As the following example shows,
one can get similar results for $r\geq 2$. 

\begin{example}
We know that if $f,g\in k\llbracket x,y\rrbracket$ are nonzero elements in the maximal ideal of
$k\llbracket x,y\rrbracket$, then $\LCT(f,g)\subseteq [0,1]^2$. In fact, we have
$\LCT(f,g)=[0,1]^2$ if and only if after a change of variables $(f,g)=(x,y)$, and otherwise
$$\LCT(f,g)\subseteq \{(\lambda_1,\lambda_2)\in [0,1]^2\mid 
\lambda_1+\lambda_2\leq 3/2\}.$$ 
Indeed, it follows from 
Example~\ref{2_0} that $\LCT(x,y)=[0,1]^2$. If there is no change of variable such that
$(f,g)=(x,y)$, then there is a line in the tangent space at the origin to 
$X=\Spec(k\llbracket x,y\rrbracket)$ 
that is contained in $T_0(V(f))\cap T_0(V(g))$. This corresponds to a point $p$ on the exceptional divisor $E$ in the blow-up $B={\rm Bl}_0(X)\overset{\pi}\to X$, and the condition says that 
$\ord_p(\pi^*(f))$, $\ord_p(\pi^*(g))\geq 2$. It follows that if $F$ is the exceptional divisor
on the blow-up of $B$ at $p$, then for every $(\lambda_1,\lambda_2)\in\LCT(f,g)$
we have
$$2\lambda_1+2\lambda_2\leq \lambda_1\cdot \ord_F(f)+\lambda_2\cdot\ord_F(g)\leq \ord_F(K_{-/X})+1=3.$$
Example~\ref{example3} a) shows that there are $f$ and $g$
such that $\LCT(f,g)=\{(\lambda_1,\lambda_2)\in [0,1]^2\mid \lambda_1+\lambda_2\leq 3/2\}$.

We note that if $r\geq 3$, then 
\begin{equation}\label{last}
\LCT(f_1,\ldots,f_r)\subseteq\{(\lambda_1,\ldots,\lambda_r)\in [0,1]^r
\mid \lambda_1+\cdots+\lambda_r\leq 2\}
\end{equation}
for every nonzero $f_1,\ldots,f_r\in (x,y)$. Indeed, we see by considering the
exceptional divisor $E$ on $B$ above
that if $\lct(f_1^{\lambda_1}\cdots f_r^{\lambda_r})\geq 1$, then 
$\sum_i\lambda_i\leq \sum_i\lambda_i\cdot\ord_E(f_i)\leq 2$.
We also observe that
if $f_1,\ldots,f_r$ are general linear forms, then $\pi\colon B\to X$ gives a log resolution of
$(X,(f_1\cdots f_r))$, and we see that in this case we have equality in (\ref{last}).
\end{example}

\providecommand{\bysame}{\leavevmode \hbox \o3em
{\hrulefill}\thinspace}

\end{document}